\documentclass[12pt,a4paper]{amsart}
\usepackage{amssymb,color}
\usepackage[english]{babel}
\usepackage{verbatim,here}
\usepackage[T1]{fontenc}
\usepackage{epstopdf}
\usepackage{floatflt,graphicx}
\usepackage{color,xcolor}
\usepackage{enumerate}
\usepackage{animate}
\usepackage{a4wide}
\usepackage{amsmath, amssymb}
\usepackage{amsthm}
\usepackage{float}

\newcounter{minutes}
\divide\time by 60
\newcounter{hours}
\multiply\time by 60 \addtocounter{minutes}{-\time}
\dedicatory{}
\commby{}
\theoremstyle{plain}
\newtheorem{thm}[equation]{Theorem}

\theoremstyle{definition}

\theoremstyle{remark}
\newtheorem{rem}[equation]{Remark}

\newtheorem{nonsec}[equation]{}

\numberwithin{equation}{section}

\newcommand{\beq}{\begin{equation}}
\newcommand{\eeq}{\end{equation}}
\newcommand{\ben}{\begin{enumerate}}
\newcommand{\een}{\end{enumerate}}
\newcommand{\bequu}{\begin{eqnarray*}}
\newcommand{\eequu}{\end{eqnarray*}}
\newcommand{\bequ}{\begin{eqnarray}}
\newcommand{\eequ}{\end{eqnarray}}
\newcommand{\B}{\mathbb{B}^2}

\newcommand{\sh}{\,\textnormal{sh}}
\newcommand{\ch}{\,\textnormal{ch}}

\begin{document}
\def\thefootnote{}

\title[]{On cyclic quadrilaterals \\ in euclidean and hyperbolic geometries}
\author[G. Wang]{Gendi Wang}
\address{%Department of Mathematical Sciences,
School of Science, Zhejiang Sci-Tech University,
         Hangzhou 310018, China}
\email{gendi.wang@zstu.edu.cn}
\author[M. Vuorinen]{Matti Vuorinen}
\address{Department of Mathematics and Statistics, University of Turku,
        Turku 20014, Finland}
\email{vuorinen@utu.fi}
\author[X. Zhang]{Xiaohui Zhang}
\address{%Department of Mathematical Sciences,
School of Science, Zhejiang Sci-Tech University,
       Hangzhou 310018, China}
\email{xiaohui.zhang@zstu.edu.cn}

\date{}

%%%%%%%%%%%%%%%%%%%%%%%%%%%%%%%%%%%%%
\begin{abstract}
Four points ordered in the positive order on the unit circle determine the vertices of a quadrilateral,
which is considered either as a euclidean or as a hyperbolic quadrilateral depending on whether the lines connecting the vertices are euclidean or hyperbolic lines.
In the case of hyperbolic lines, this type of quadrilaterals are called ideal quadrilaterals.
Our main result gives a euclidean counterpart of an earlier result on the hyperbolic distances between the opposite sides of ideal quadrilaterals.
The proof is based on computations involving hyperbolic geometry.
We also found a new formula for the hyperbolic midpoint of a hyperbolic geodesic segment in the unit disk.
As an application of some geometric properties, we provided a euclidean construction of the symmetrization of random four points on the unit circle with respect to a diameter which preserves the absolute cross ratio of quadruples.
\end{abstract}

%%%%%%%%%%%%%%%%%%%%%%%%%%%%%%%%%%%%%

\maketitle

\footnotetext{\texttt{{\tiny File:~\jobname .tex, printed: \number\year-%
\number\month-\number\day, \thehours.\ifnum\theminutes<10{0}\fi\theminutes}}}
\makeatletter

\makeatother

% Text of article.

%%%%%%%%%%%%%%%%%%%%%%%%%%%%%%%%%%%%%
%%%%%%%%%%%%%%%%%%%%%%%%%%%%%%%%%%%%%
%%%%%%%%%%%%%%%%%%%%%%%%%%%%%%%%%%%%%
\section{Introduction}

In complex analysis, quadruples of points have a very special role:
the absolute cross ratio of four points $a\,,b\,,c\,, d$ in the complex plane $\mathbb{C}$,
\begin{equation}\label{myaxrat}
|a,b,c,d| = \frac{|a-c||b-d|}{|a-b||c-d|}
\end{equation}
is preserved under M\"obius transformations.
This fact is of fundamental importance to geometry and, in particular,
to the hyperbolic geometry of the unit disk $ \mathbb{B}^2$ and of the upper half plane $ \mathbb{H}^2\,.$
Indeed, the hyperbolic metric of both domains can be defined in terms of the absolute cross ratio.

Our goal is to study ordered quadruples $(a,b,c,d)$ of points on the unit circle,
points being listed  in the  order they occur when we traverse the unit circle in the positive direction.
These points are vertices of a quadrilateral.
We analyse the points of intersection of the lines through opposite sides of the quadrilateral and interpret our observations in terms of the euclidean or the hyperbolic geometry.
%The classical facts about triangles will be useful here because our work includes study of triangles with vertices taken from the set of these points of intersection.

In the course of this work we discover a number of formulas which are very useful for our work and which,
as far as we know, are new --- at least we have not found them in literature.
We make frequent use of the well-known formula \eqref{LIS} which gives an expression for the point of intersection
$LIS[a,b,c,d]$ of the two lines $L[a,b]$ and $L[c,d]$, through $a\,,b$ and $c,d\,,$ respectively.

Quadrilaterals in $\B$ whose vertices are on the unit circle $\partial \B$ and whose boundary consists of four circular arcs,
orthogonal to the unit circle, are called ideal hyperbolic quadrilaterals.
An ideal hyperbolic quadrilateral can be divided into four Lambert quadrilaterals.
In \cite{vw2, w}, the authors studied the bounds of the two adjacent sides of a hyperbolic Lambert quadrilateral in the unit disk.
In particular, the following sharp inequalities hold for the product and the sum of the hyperbolic distances
between the opposite sides of an ideal hyperbolic quadrilateral \cite{vw2}.

Let $J^*[a,b]$ denote the hyperbolic line through two points $a,b\in \partial\B$.
For $A, B \subset \B$, let
$$d_{\rho}(A,B)=\inf_{x\in A, y\in B}\rho_{\mathbb{B}^2}(x,y),$$
where $\rho_{\B}$ is the hyperbolic metric in the unit disk defined as \eqref{myrho}.

%===============================================================================
\begin{thm}\label{dd}\cite[Corollary 1.4]{vw2}
Let $Q(a,b,c,d)$ be an ideal hyperbolic quadrilateral in $\B$.
Let $d_1=d_\rho(J^*[a,d],J^*[b,c])$ and $d_2=d_\rho(J^*[a,b], J^*[c,d])$ (see Figure~\ref{idealfigb}).
Then
\begin{eqnarray*}
d_1d_2\leq (2\log(\sqrt{2}+1))^2
\end{eqnarray*}
and
\begin{eqnarray*}
d_1+d_2\geq 4\log(\sqrt{2}+1)\,.
\end{eqnarray*}
In both cases equalities hold if and only if $|a,b,c,d|=2$.
\end{thm}
%===============================================================================

We will analyse this further and our discoveries here consist of finding  counterparts  of these formulas for quadrilaterals
obtained from ideal quadrilaterals by replacing the sides, circular arcs, with the euclidean segments with the same endpoints.

%==========================================================================
\begin{figure}[h]
\begin{minipage}[t]{0.45\linewidth}
\centering
\includegraphics[width=8cm]{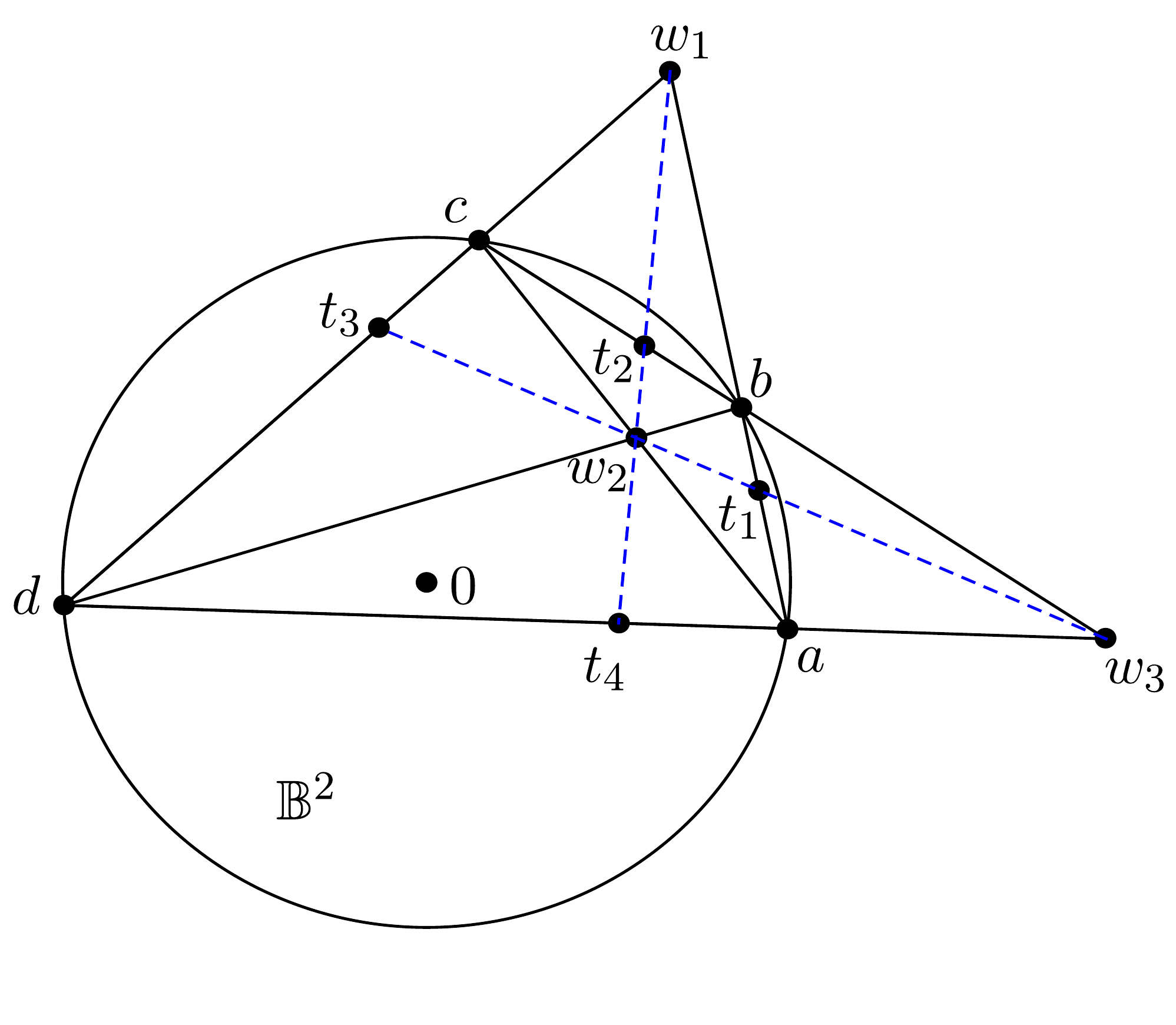}
\caption{\label{cyclicfig12}}
\end{minipage}
\hfill
\begin{minipage}[t]{0.45\linewidth}
\centering
\includegraphics[width=6cm]{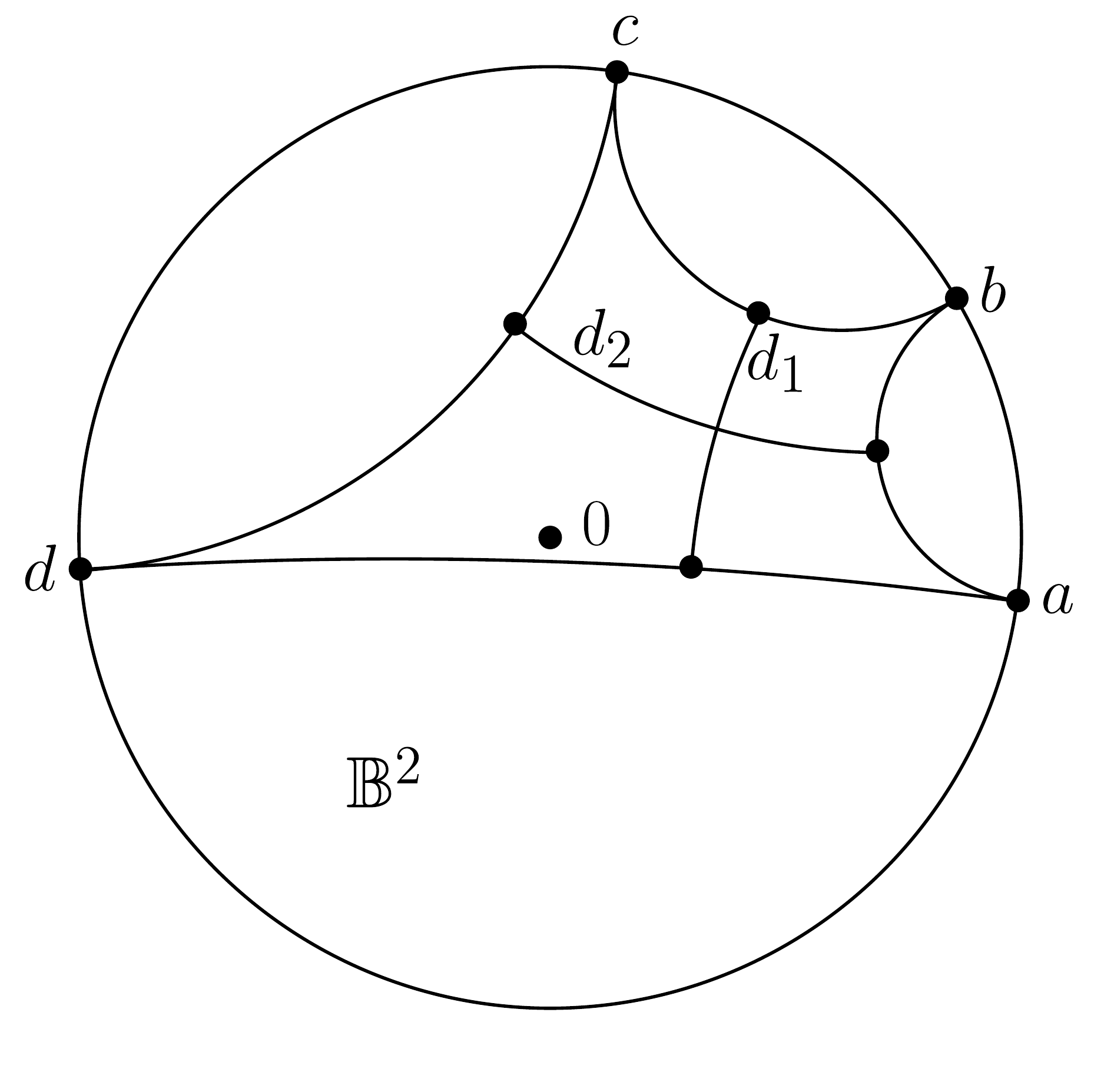}
\caption{\label{idealfigb}}
\end{minipage}
\end{figure}
%==========================================================================

%===========================================================================================
\begin{thm} \label{euclamb}
Let $w_1= LIS[a,b,c,d]\,,$ $w_2= LIS[a,c,b,d]\,,$ $w_3= LIS[a,d,b,c]\,.$
Let $t_1=L[w_3,w_2] \cap L[a,b]\,,$  $t_2=L[w_1,w_2] \cap L[b,c]\,,$ $t_3=L[w_3,w_2] \cap L[c,d]\,,$ $t_4=L[w_1,w_2] \cap L[a,d]\,$ (see Figure~\ref{cyclicfig12}).
Let
$d_1=d_{\rho}(J^*[a,d],J^*[b,c])\,,$ $d_2=d_{\rho}(J^*[a,b],J^*[c,d])\,$ (see Figure~\ref{idealfigb}).
Then
$$\frac{{\rm sh}\,\frac12\rho_{\B}(w_2,t_2)+{\rm sh}\,\frac12\rho_{\B}(w_2,t_4)}{{\rm sh}\,\frac12\rho_{\B}(t_2,t_4)}={\rm th}\,\frac{d_1}2\,,$$
$$\frac{{\rm sh}\,\frac12\rho_{\B}(w_2,t_1)+{\rm sh}\,\frac12\rho_{\B}(w_2,t_3)}{{\rm sh}\,\frac12\rho_{\B}(t_1,t_3)}={\rm th}\,\frac{d_2}2\,,$$
and
$${\rm th}^2\frac{d_1}2+{\rm th}^2\frac{d_2}2=1.$$
\end{thm}
%===========================================================================================

The bisection problem in the classical hyperbolic geometry has been studied in \cite[Construction 3.1]{g} and \cite[2.9]{kv}.
Recently, Vuorinen and Wang \cite{vw1} provided several geometric constructions based on euclidean compass and ruler
to find the hyperbolic midpoint of a hyperbolic geodesic segment in the unit disk or in the upper half plane.
In this paper, we give an explicit formula for the hyperbolic midpoint of a hyperbolic geodesic segment in the unit disk.

%===========================================================================================
\begin{thm}\label{myhmidp}
For given $x,y \in \B\,,$ the hyperbolic midpoint $z\in \B$ with
$\rho_{\B}(x,z)=\rho_{\B}(y,z)=\rho_{\B}(x,y)/2$ is given by
\begin{equation}\label{myzformua}
z= \frac{y(1-|x|^2) + x(1-|y|^2)}{1-|x|^2|y|^2 + A[x,y] \sqrt{(1-|x|^2)(1-|y|^2)}}
\end{equation}
where $A[x,y]$ is the Ahlfors bracket defined as \eqref{myahl}.
\end{thm}
%===========================================================================================

%%%%%%%%%%%%%%%%%%%%%%%%%%%%%%%%%%%%%
%%%%%%%%%%%%%%%%%%%%%%%%%%%%%%%%%%%%%
%%%%%%%%%%%%%%%%%%%%%%%%%%%%%%%%%%%%%
\section{Preliminary notation}
%%%%%%%%%%%%%%%%%%%%%%%%%%%%%%%%%%%%%
%%%%%%%%%%%%%%%%%%%%%%%%%%%%%%%%%%%%%
%%%%%%%%%%%%%%%%%%%%%%%%%%%%%%%%%%%%%

We will give here some formulas about the geometry of lines and triangles on which our later work is based.
In particular, Euler's formula for the orthocenter of a triangle is very useful.

%===========================================================================================
\begin{nonsec}{\bf Geometry and complex numbers.}
The extended complex plane $\overline{\mathbb{C}}= {\mathbb{C}} \cup \{\infty\}$ is identified with the Riemann sphere via the stereographic projection.
The stereographic projection then can be used to define the chordal distance $q(x,y)$ between points  $x, y \in\overline{\mathbb{C}}$ \cite{b}.
If $f: (G_1, d_1)\to (G_2,d_2)$ is a homeomorphism between metric spaces, then the Lipschitz constant of $f$ is defined by
\begin{equation}\label{myLip1}
\mbox{Lip}(f) = \sup\left\{  \frac{d_2(f(x),f(y))}{d_1(x,y)} \,: x,y \in G_1, x\neq y\right\}.
\end{equation}

%The dot product of two complex numbers is
%\begin{equation}\label{dot}
%z \cdot w = {\rm{Re}}(z \overline{w})= (z \overline{w} +\overline{z} w)/2\,.
%\end{equation}
Let $L[a,b]$ stand for the line through $a$ and $b\,(\neq a)\,.$
For distinct points $a,b,c,d \in {\mathbb{C}}$ such that the lines $L[a,b]$ and  $L[c,d]$ have a unique
point $w$ of intersection, let
$$
w=LIS[a,b,c,d] = L[a,b]\cap L[c,d] \,.
$$
This point is given by
\begin{equation}{\label{LIS}}
w=LIS[a,b,c,d] =\frac{u}{v},
\end{equation}
where
\begin{equation}\label{Lineint}
\begin{cases}
{\displaystyle
{u= (\overline{a}b -a  \overline{b})(c-d)-(a-b) ( \overline{c}d -c  \overline{d}) }}&\\
{\displaystyle {v = ( \overline{a}- \overline{b})(c-d)-(a-b) (\overline{c} - \overline{d})\,.}}&
\end{cases}
\end{equation}

Let $C[a,b,c]$  be the   circle  through distinct non-collinear points $\,a,\,b,\, c\,.$
The formula \eqref{LIS} gives easily the formula for the center $m(a,b,c)$ of $C[a,b,c]\,.$
For instance, we can find two points on the bisecting normal to the side $[a,b]$ and another
two points on the bisecting normal to the side $[a,c]$ and then apply \eqref{LIS}  to get
$m(a,b,c)\,.$
In this way we see that the center $m(a,b,c)$ of  $C[a,b,c]$  is
\begin{equation}\label{mfun}
m(a,b,c)=\frac{ |a|^2(b-c) +  |b|^2(c-a) +  |c|^2(a-b) }
{a(\overline{c}-\overline{b}) +b(\overline{a}-\overline{c})+ c(\overline{b}-\overline{a})}\,.
\end{equation}

Euler's formula gives the orthocenter $o(a,b,c)$ of a triangle with vertices $a\,,b\,,c$ as
\begin{align}\label{eul}
o(a,b,c)
&=a+b+c-2 m(a,b,c)\nonumber\\
&=\frac{\overline{a}(b+c-a)(b-c) +\overline{b}(c+a-b)(c-a) + \overline{c}(a+b-c)(a-b)}{\overline{a}(b-c) +\overline{b}(c-a) + \overline{c}(a-b)}\,.
\end{align}
\end{nonsec}
%===========================================================================================

%===========================================================================================
\begin{nonsec}{\bf M\"obius transformations.}\label{mymob}
A M\"obius transformation is a mapping of the form
$$z \mapsto \frac{az+b}{cz+d}\,, \quad a,b,c,d,z \in {\mathbb C}\,, ad-bc\neq 0\,.$$
The special M\"obius transformation
\begin{equation}\label{myTa}
T_a(z) = \frac{z-a}{1- \overline{a}z}\,, \quad a \in \B \setminus \{0\}\,,
\end{equation}
maps the unit disk $\mathbb{B}^2$ onto itself with $T_a(a) =0\,.$
Its Lipschitz constant \eqref{myLip1} as a mapping $T_a: \B\to \B$ with respect to the euclidean metric is \cite[p.\,43]{b}
\begin{equation}\label{myLip2}
\mbox{Lip}(T_a)= \frac{1+ |a|}{1-|a|}\,.
\end{equation}

We sometimes use the notation $x^* = x/|x|^2 = 1/\overline{x}$ for $x\in {\mathbb C} \setminus \{0\}\,.$
\end{nonsec}
%===========================================================================================

%===========================================================================================
\begin{nonsec}{\bf Hyperbolic geometry.}
We recall some basic formulas and notation for hyperbolic geometry from \cite{b}.
The hyperbolic metric $\rho_{\mathbb{B}^2}$ is defined by
\begin{equation}\label{myrho}
\sh \frac{\rho_{\mathbb{B}^2}(x,y)}{2}= \frac{|x-y|}{\sqrt{(1-|x|^2)(1-|y|^2)}}\,.
\end{equation}
Since
\begin{align*}
|x,x^*,y^*,y|=\frac{|x-y^*||y-x^*|}{|x-x^*||y-y^*|}
%&=\frac{\left|x|y|-\frac y{|y|}\right|\left|y|x|-\frac x{|x|}\right|}{(1-|x|^2)(1-|y|^2)}\\
%&=\frac{1+|x|^2|y|^2-2 x\cdot y}{(1-|x|^2)(1-|y|^2)}\\
&=1+\frac{|x-y|^2}{(1-|x|^2)(1-|y|^2)}\,,
\end{align*}
another form of the formula for the hyperbolic metric follows
$$\ch \frac{\rho_{\mathbb{B}^2}(x,y)}{2}=\sqrt{|x,x^*,y^*,y|}\,.$$

We also use the Ahlfors bracket notation $A[x,y]$ \cite[7.37]{avv}, for $x,y \in {\mathbb{B}^2}$
\begin{equation}\label{myahl}
A[x,y]^2 =(1-|x|^2)(1-|y|^2) + |x-y|^2\,.
\end{equation}

The circle which is orthogonal to the unit circle and contains two distinct points $x,y\in \mathbb{C}$ is denoted by $C[x,y]$.
If $x,y\in {\mathbb{B}^2}$ are distinct points, then $C[x,y]\cap \partial {\mathbb{B}^2} = \{x_*, y_*\}$
where the points are labelled in such a way that $x_*,x,y,y_*$ occur in this order on  $C[x,y]\,.$
We denote by $J[x,y]$ the hyperbolic geodesic segment joining two distinct points $ x,y \in {\mathbb{B}^2}\,.$
Then $J[x,y] $ is a subarc of
%$ C[x,y]\cap{\mathbb{B}^2}\,$ and
the hyperbolic line $J^*[x_*,y_*]= C[x,y]\cap{\mathbb{B}^2}$ .
\end{nonsec}
%===========================================================================================

%%%%%%%%%%%%%%%%%%%%%%%%%%%%%%%%%%%%%
%%%%%%%%%%%%%%%%%%%%%%%%%%%%%%%%%%%%%
%%%%%%%%%%%%%%%%%%%%%%%%%%%%%%%%%%%%%
\section{Geometric Observations}\label{go}
%%%%%%%%%%%%%%%%%%%%%%%%%%%%%%%%%%%%%
%%%%%%%%%%%%%%%%%%%%%%%%%%%%%%%%%%%%%
%%%%%%%%%%%%%%%%%%%%%%%%%%%%%%%%%%%%%

Let  $a,b,c,d \in \partial  \B\,$ be points listed in the order they occur when one traverses the unit circle in the positive direction.
%In the following for clarity we always assume that
%$$0<{\rm Arg}[a]<{\rm Arg}[b]<{\rm Arg}[c]<{\rm Arg}[d]< 2 \pi\,.$$

%===========================================================================================
\begin{thm}\label{myoarcisec}
%Let  $a,b,c,d \in \partial  \B\,,$ be points listed in the order they occur when one traverses the unit circle in the positive direction.

(1) Let
$$w_1= LIS[a,b,c,d]\,,w_2= LIS[a,c,b,d]\,,w_3= LIS[a,d,b,c]\,.$$
Then the point $w_2$ is the orthocenter of the triangle with vertices at the points $0$, $w_1$, $w_3.$

(2)
%Let $w=C[a,c]\cap C[b,d] \in \B\,.$
%The orthogonal arcs, intersecting the unit circle at $a,c$ and $b,d\,,$ resp., have a unique point of intersection given by
The point of intersection of the hyperbolic lines $J^*[a,c]$ and $J^*[b,d]$ is given by
\begin{equation*}\label{myoaisec}%w
Y= \frac{(ac-bd)\pm\sqrt{(a-b)(b-c)(c-d)(d-a)}}{a-b+c-d}\,,
\end{equation*}
where the sign "$+$" or "$-$" in front of the square
%corresponds to the value with absolute value less than
root is chosen such that $|Y|<1\,.$
 \end{thm}
%===========================================================================================
%===========================================================================================
\begin{proof}
(1) %Let $LIS[a,b,c,d]=w_1$, $LIS[a,c,b,d]=w_2$, $LIS[a,d,b,c]=w_3$.
Since all the points $a,b,c,d$ are unimodular, i.e., $a \overline{a}=b \overline{b}=c \overline{c}=d \overline{d}=1\,,$
\eqref{LIS} is simplified to
\begin{equation}\label{ULIS}
w_1=L[a,b]\cap L[c,d]=\frac{ab(c+d)-cd(a+b)}{ab-cd}  \,.
\end{equation}
This follows easily if we multiply $u$ and $v$ in \eqref{Lineint} by the product $abcd$ and use unimodularity.

From \eqref{ULIS} there follows that
\begin{equation}\label{myabcd}
w_2=L[a,c]\cap L[b,d]=\frac{ac(b+d)-bd (a+c)}{ac-bd}
\end{equation}
and
\begin{equation}\label{myadbc}
w_3=L[a,d]\cap L[b,c]=\frac{ad(b+c)-bc(a+d)}{ad-bc} \,.
\end{equation}

These formulas together with \eqref{eul} yield
\begin{equation*}\label{myo}
o(0,w_1,w_3)=\frac{\overline{w_1}{w_3}+{w_1}\overline{w_3}}{\overline{w_1}{w_3}-{w_1}\overline{w_3}}(w_3-w_1)=w_2\,.
\end{equation*}
Therefore, the point $w_2$ is the orthocenter of the triangle with  vertices at the points $0$, $w_1$, $w_3\,$, see Figure~\ref{cyclicfig31}.

%==========================================================================
\begin{figure}[H]
\centering
\includegraphics[width=8.2cm]{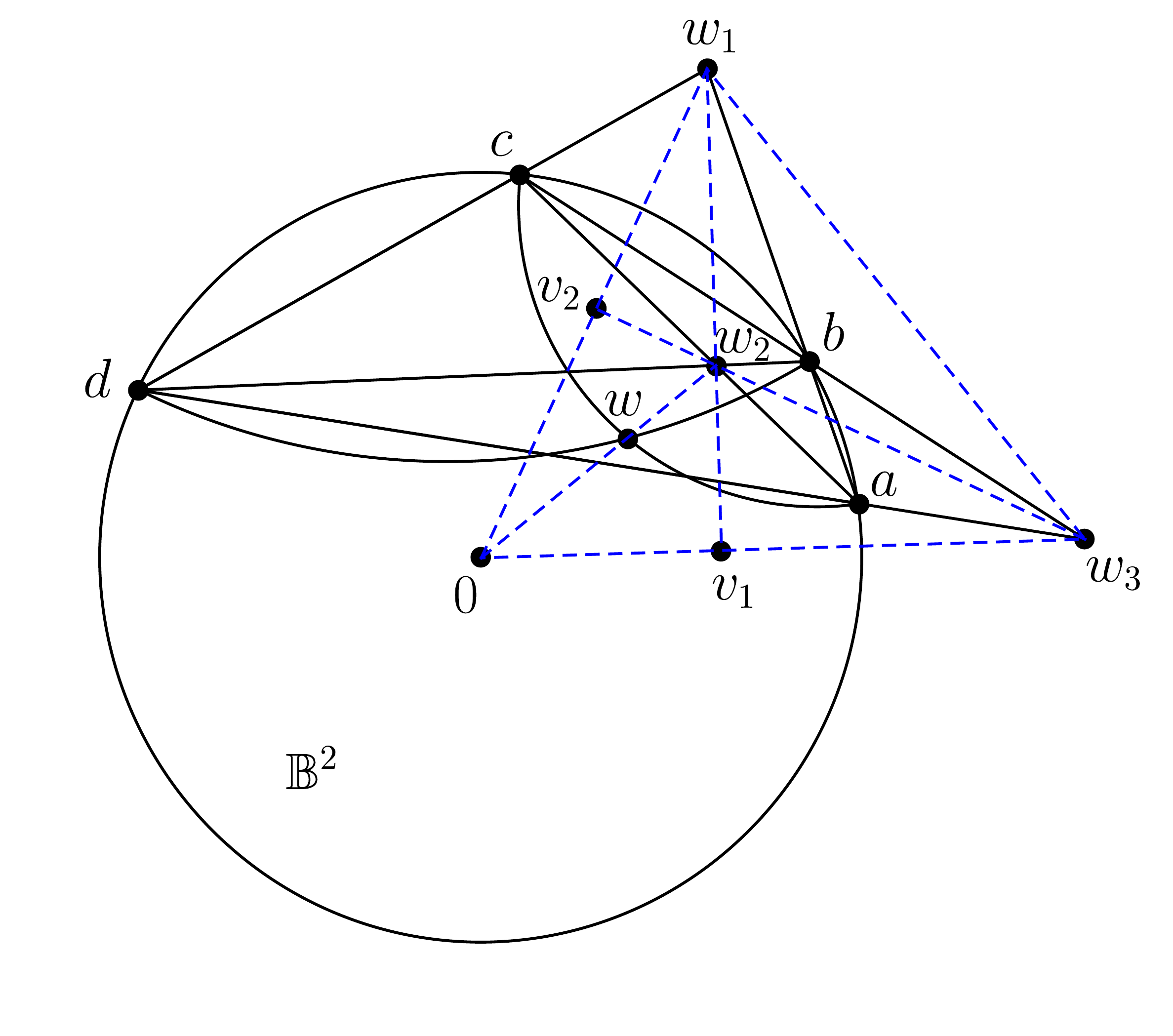}
\caption{\label{cyclicfig31}}
\end{figure}
%==========================================================================

(2)
%Let $C[a,c],\,C[b,d]$ be the orthogonal arcs, intersecting the unit circle at $a,c$ and $b,d\,,$ resp.
Let $z=C[a,c]\cap C[b,d]$ and $m(a,c)$ be the center of $C[a,c]$.
Then
$$\left|\frac{a+c}2\right| |m(a,c)|=1$$
and hence
\begin{equation}\label{myOac}
m(a,c)=\frac{2(a-c)}{a\overline{c}-\overline{a}c}=\frac{2 ac}{a+c}\,.
\end{equation}
By \eqref{mfun}, we have
$$m(a,z,c)=\frac{ac(1-|z|^2)}{(a+c)-z-ac\overline{z}}\,.$$
Since $m(a,c)=m(a,z,c)$, we obtain
$$\frac{ac\overline{z}+z}{a+c}=\frac{ 1+|z|^2}{2}\,.$$

A similar argument yields
$$\frac{bd\overline{z}+z}{b+d}=\frac{ 1+|z|^2}{2}\,.$$

Therefore,
$$\overline{z}=\frac{a-b+c-d}{ac(b+d)-bd(a+c)}z$$
and the point $z$ satisfies the following quadratic equation
\begin{equation}\label{mysympt}
%w^2-2\frac{ac-bd}{a-b+c-d}w+\frac{ac(b+d)-bd(a+c)}{a-b+c-d}=0\,.
(a-b+c-d)z^2-2(ac-bd)z+(ac(b+d)-bd(a+c))=0\,.
\end{equation}
Solving this equation, we have the result and the proof is complete.
\end{proof}
%===========================================================================================

%===========================================================================================
\begin{rem}\label{myhmidps}
Denote
\begin{equation}\label{myisecpts4}
w=J^*[a,c]\cap J^*[b,d].
\end{equation}
Let $u_1=J^*[a,c]\cap J[0,w_2]$ and $u_2=J^*[b,d]\cap J[0,w_2]$.
By \cite[Proposition 3.1]{vw2}, $u_1$ is the hyperbolic midpoint of $J[0,w_2]$ and so is $u_2$.
Hence $u_1=u_2=w$, i.e., the point $w$, intersection of the two hyperbolic lines $J^*[a,c]$ and $J^*[b,d]$,
is the hyperbolic midpoint of the hyperbolic geodesic segment $J[0,w_2]$,
where $w_2$ is the point of intersection of the two euclidean lines $L[a,c]$ and $L[b,d]$.
\end{rem}
%===========================================================================================

%The circle through three non-collinear points $a,b,c$ is denoted by $C[a,b,c]\,.$
%===========================================================================================
\begin{thm}\label{myoarcisec2}
% Let $a,b,c,d \in \partial  \B\,,$ be points as in Theorem \ref{myoarcisec}.
 %and $w_2$ the point in Theorem \ref{myoarcisec} (1).Theorem \ref{myoarcisec} (1)
Let $w_2$ be as \eqref{myabcd}.
The centers of the four circles through the point $w_2$
$$ C[a,b,w_2], C[b,c,w_2], C[c,d,w_2], C[a,d,w_2]$$
form the vertices of a rhomboid.
Moreover, the euclidean center of the rhomboid is the euclidean midpoint of the segment $[0,w_2]$.
\end{thm}
%===========================================================================================
%===========================================================================================
\begin{proof}%{\bf Proof of Theorem  \ref{myoarcisec2}.}
Let
$$p_1=m[a,b,w_2],\, p_2=m[b,c,w_2],\, p_3=m[c,d,w_2],\, p_4=m[a,d,w_2]\,.$$
Clearly, the lines $L[p_1,p_2]$ and $L[p_3,p_4]$ are the bisecting normals to the segments $[b,w_2]$ and $[d,w_2]$, respectively,
and hence $L[p_1,p_2]$ is parallel to $L[p_3,p_4]$.
A similar argument yields that $L[p_1,p_4]$ is parallel to $L[p_2,p_3]$.
Therefore, $p_1,\,p_2,\,p_3,\,p_4$ form the vertices of a rhomboid.

Let
$$p_5=LIS[p_1,p_3,p_2,p_4]\,.$$
By \eqref{mfun}, we have
$$p_1=m(a,b,w_2)=\frac{ab(c-d)}{ac-bd}\,\,\,{\rm and}\,\,\,p_3=m(c,d,w_2)=\frac{cd(a-b)}{ac-bd}\,.$$
Hence
$$p_5=\frac12(p_1+p_3)=\frac12 w_2\,.$$
Namely, $p_5$ is the euclidean midpoint of the segment $[0,w_2]$, see Figure~\ref{cyclicfig39b}.
\end{proof}
%===========================================================================================

%==========================================================================
\begin{figure}[h]
\centering
\includegraphics[width=8.2cm]{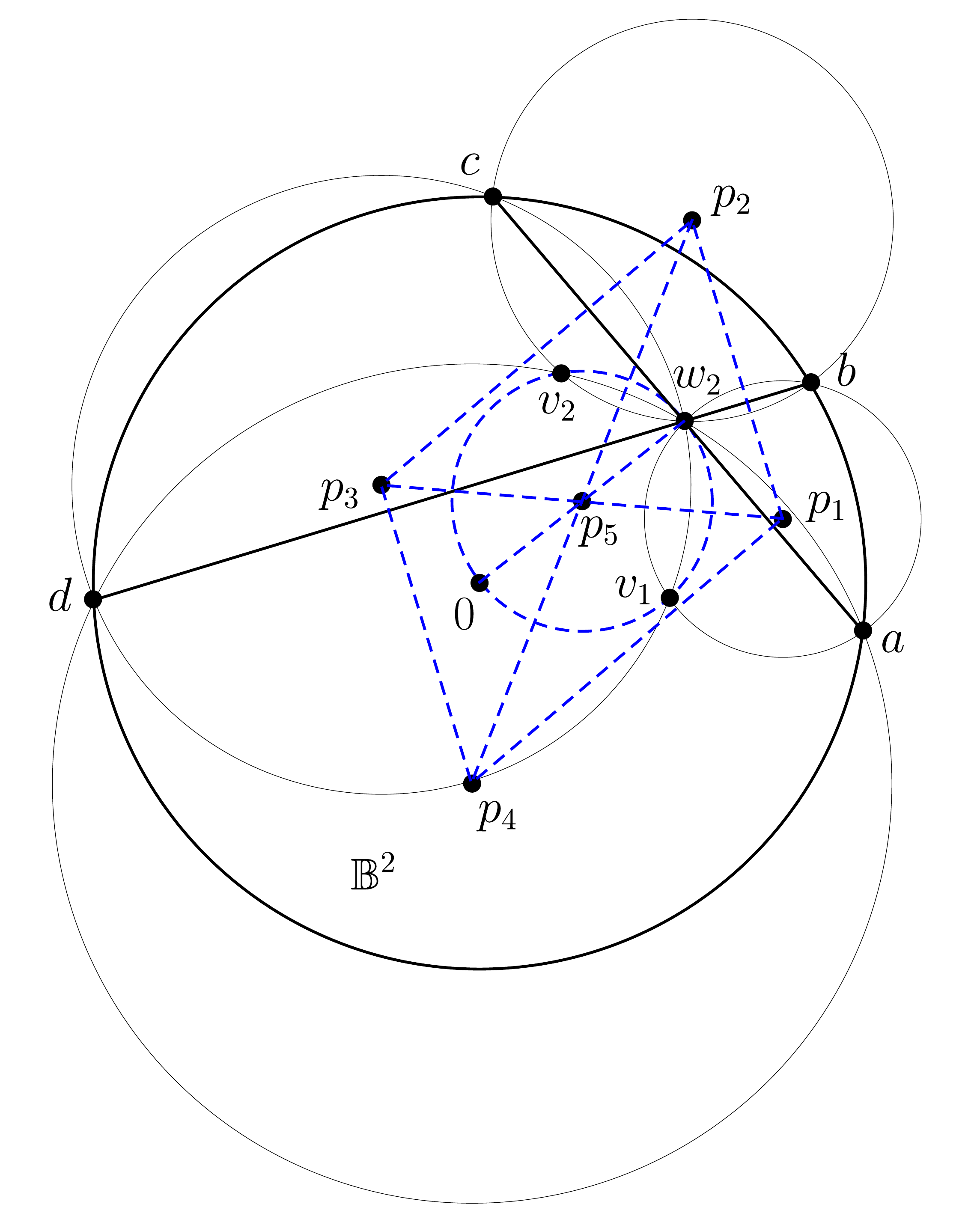}
\caption{\label{cyclicfig39b}  }
\end{figure}
%==========================================================================

%===========================================================================================
\begin{rem}\label{myuv}
Let $w_1\,,w_3$ be as \eqref{ULIS} and \eqref{myadbc}, respectively.
Denote
\begin{equation*}\label{myisecpts3}
\{v_1, w_2\}= C[a,b,w_2]\cap C[c,d,w_2]\,,
\end{equation*}
\begin{equation*}\label{myisecpts2}
\{v_2, w_2\}=C[b,c,w_2]\cap C[a,d,w_2]\,.
\end{equation*}
Since $p_5$ is on $L[p_1,p_3]$ and $L[p_2,p_4]$ which are the bisecting normals to the segments  $[v_1,w_2]$ and $[v_2,w_2]$, respectively, by Theorem \ref{myoarcisec2}
$$|p_5-v_1|=|p_5-v_2|=|p_5-w_2|=|p_5-0|=\frac12|w_2|\,.$$
Hence $0, v_1, w_2, v_2$ are on the circle centered at $p_5$ with diameter $[0,w_2]$.

By intersecting-secants theorem, $w_3\,,w_2\,,v_2$ are collinear and so are $w_1\,,w_2\,,v_1$, see Figure~\ref{cyclicfig31} and Figure~\ref{cyclicfig39b}.
Since $[0,v_2]$ is perpendicular to $[w_2,v_2]$, together with Theorem \ref{myoarcisec}(1),
% is the orthocenter of the triangle $\Delta(0,w_1,w_3)\,,$
we have that $0,v_2,w_1$ are collinear.
A similar argument yields that $0,v_1,w_3$ are collinear.
\end{rem}
%===========================================================================================

%===========================================================================================
\begin{thm}\label{myoccen}%in Theorem \ref{myoarcisec}
Let $w_1\,,w_3\,,w$ be as \eqref{ULIS}, \eqref{myadbc} and \eqref{myisecpts4}, respectively.
Then the circle orthogonal to the unit circle goes through $w$ if and only if the center of the circle is on the line $L[w_1,w_3]$.
\end{thm}
%===========================================================================================
%===========================================================================================
\begin{proof}
Let $c_1=m(a,c)$ and $c_2=m(b,d)$.
%Then $c_1\,,w_1\,,c_2\,,w_3$ are collinear.
Apparently, the circle orthogonal to the unit circle goes through $w$
if and only if the center of the circle is on the bisecting normal $L[c_1,c_2]$ to the segment $[w,1/\overline{w}]$.
It suffices to show that the four points $c_2$, $w_1$, $c_1$, $w_3$ are collinear, see Figure~\ref{cyclicfig311b}.
%$L[w_1,w_2]=L[c_1,c_2]$.

By \eqref{ULIS} and \eqref{myOac},
%$$m(a,c)=\frac{2ac}{a+c}\,\,\,{\rm and}\,\,\,m(b,d)=\frac{2bd}{b+d}\,.$$
$$\frac{c_2-c_1}{w_1-c_1}=\frac{2(ab-cd)}{(a-c)(b+d)}\,.$$
Since $a\overline a=b\overline b=c\overline c=d\overline d=1$, we have
$$\overline{\left(\frac{ab-cd}{(a-c)(b+d)}\right)}=\frac{abcd(\overline a\overline b-\overline c\overline d)}{abcd(\overline a-\overline c)(\overline b+\overline d)}
%=\frac{cd-ab}{(c-a)(b+d)}
=\frac{(ab-cd)}{(a-c)(b+d)}\,.$$
Therefore, $\frac{c_2-c_1}{w_1-c_1}\in \mathbb{R}$ and hence $w_1\in L[c_1,c_2]$.
A similar argument yields $w_3\in L[c_1,c_2]$. Hence, the four points $c_2$, $w_1$, $c_1$, $w_3$ are collinear.
%Clearly, $L[c_1,c_2]$ is the bisecting normal to the  segment $[w,1/\overline{w}]$.
%Namely, $L[w_1,w_3]$ is also the bisecting normal to the  segment $[w,1/\overline{w}]$.

%==========================================================================
\begin{figure}[h]
\centering
\includegraphics[width=8.2cm]{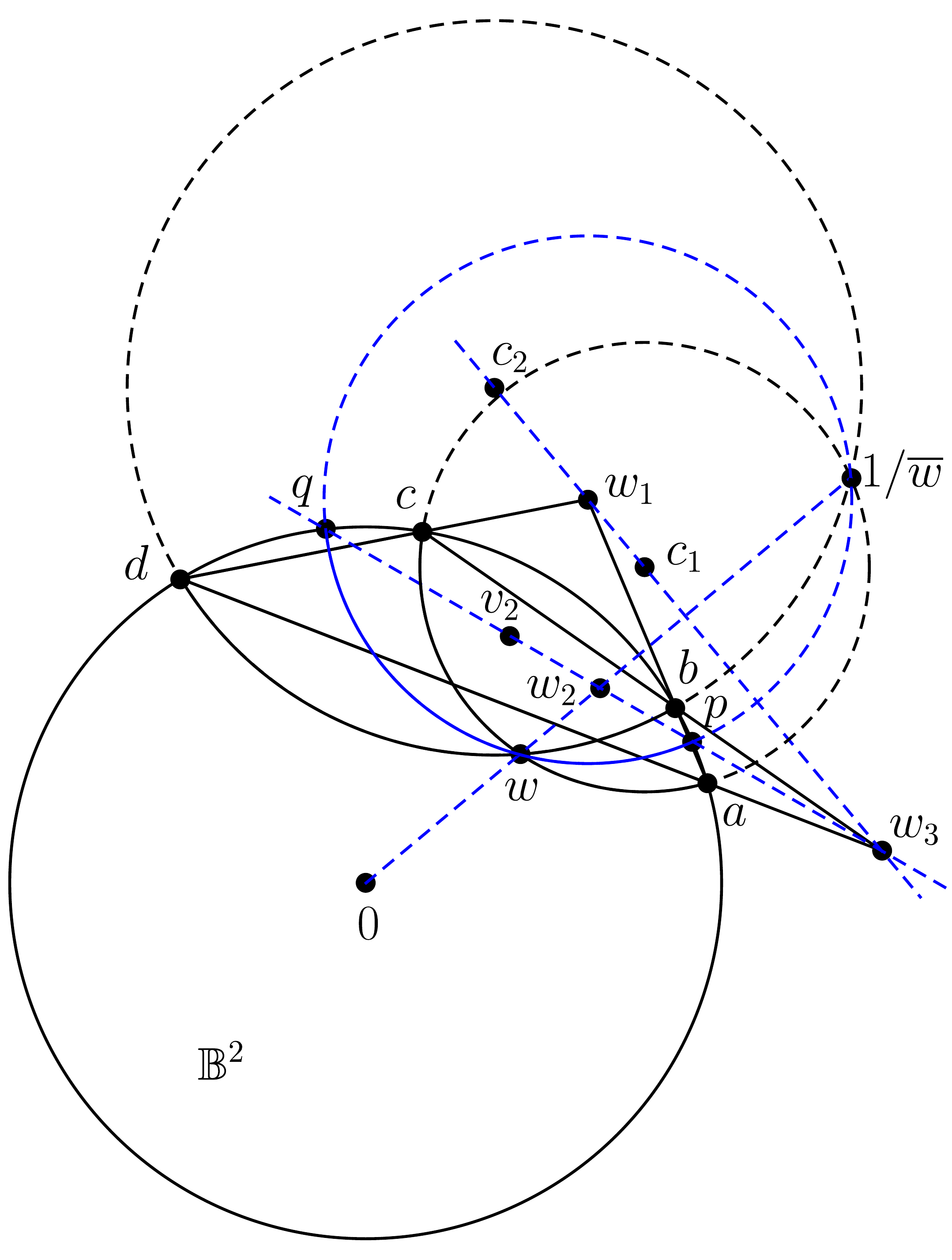}
\caption{\label{cyclicfig311b}  }
\end{figure}
%==========================================================================

%It is easy to see that the centers of the circles through $w$ orthogonal to the unit circle are all located  on the bisecting normal  $L[w_1,w_2]$ to the segment %$[w,1/\overline{w}]$.

%Let $C$ be any circle orthogonal to the unit circle with center $cen\in L[w_1,w_3]$\,.
%Since the circle centered at $cen$ with radius $|cen-w|$ is orthogonal to the unit circle and the circle satisfying the assumption is unique,
%we have $w\in C$.

%This completes the proof.
\end{proof}
%===========================================================================================

%===========================================================================================
\begin{rem}\label{myoccen1}
Let $C_o$ be the circle orthogonal to the unit circle centered at $w_1$ and
denote its points of intersection with the unit circle by
\begin{equation}\label{myisecpts5}
\{p,q\}= C_o\cap \partial \B\,.
\end{equation}
By Remark \ref{myuv}, $w_2\in[w_3,v_2]$ and $L[w_3,v_2]$ is perpendicular to $L[0,w_1]$.
By Theorem \ref{myoccen}, $w\in J^*[p,q]$. By Remark \ref{myhmidps} and \cite[Proposition 3.1]{vw2}, $w_2\in[p,q]$.
Since $L[p,q]$ is also perpendicular to $L[0,w_1]$, we see that the five points $w_3, p, w_2, v_2, q$ are collinear, see Figure~\ref{cyclicfig311b}.
\end{rem}
%===========================================================================================

%===========================================================================================
\begin{thm}\label{mythoccen}
Let $p\,,q$ be as \eqref{myisecpts5}.
The hyperbolic line $J^*[p,q]$ is the angular bisector of the angle $\angle(J^*[a,c], w, J^*[b,d])$.
\end{thm}
%===========================================================================================
%===========================================================================================
\begin{proof} Let $C_o$ be as in Remark \ref{myoccen1}.
Tangent-secant theorem yields
$$|w_1|^2-1=|w_1-q|^2=|w_1-a||w_1-b|=|w_1-c||w_1-d|\,.$$
Then $a$ and $b$ are a pair of inverse points with respect to $C_o$, and so are $c$ and $d$.
Hence the inversion in $C_o$ maps $J^*[a,c]$ onto $J^*[b,d]$ and $J^*[p,q]$ onto itself, see Figure \ref{cyclicfig311b}.

Therefore, $J^*[p,q]$ is the angular bisector of the angle $\angle(J^*[a,c], w, J^*[b,d])$ since M\"obius transformations preserve angles.
\end{proof}
%===========================================================================================

%===========================================================================================
\begin{rem}
Since $|w_1-w|=|w_1-q|$, the proof of Theorem \ref{mythoccen} yields that
$C[a,b,w]$ is tangent to $C[c,d,w]$ at the point $w$.
\end{rem}
%===========================================================================================

%%%%%%%%%%%%%%%%%%%%%%%%%%%%%%%%%%%%%
%%%%%%%%%%%%%%%%%%%%%%%%%%%%%%%%%%%%%
%%%%%%%%%%%%%%%%%%%%%%%%%%%%%%%%%%%%%
%\section{Ideal hyperbolic quadrilaterals and euclidean lines}
\section{Proofs of Main Results}
%%%%%%%%%%%%%%%%%%%%%%%%%%%%%%%%%%%%%
%%%%%%%%%%%%%%%%%%%%%%%%%%%%%%%%%%%%%
%%%%%%%%%%%%%%%%%%%%%%%%%%%%%%%%%%%%%

%===========================================================================================
\begin{nonsec}{\bf Proof of Theorem  \ref{euclamb}.}
By \cite[(4.1) in the proof of Corollary 1.4]{vw2},
$$d_1=2\,{\rm arth}\,\sqrt{\frac 1{|a,b,c,d|}}=2\,{\rm arth}\,r,\,\, d_2=2\,{\rm arth}\,\sqrt{\frac 1{|d,a,b,c|}}=2\,{\rm arth}\,r',$$
where $r^2+r'^2=1$.
Clearly,
$${\rm th}^2\frac{d_1}2+{\rm th}^2\frac{d_2}2=1.$$

By \eqref{LIS}, we have
\begin{align*}
t_2=L[w_1,w_2]\cap L[b,c]
%&=\frac{(\overline{w_1}{w_2}-{w_1}\overline{w_2})(b-c)-(w_1-w_2)(\overline{b}c-b\overline{c})}
%{(\overline{w_1}-\overline{w_2})(b-c)-(w_1-w_2)(\overline{b}-\overline{c})}\\
%&=\frac{bc(\overline{w_1}{w_2}-{w_1}\overline{w_2})+(b+c)(w_1-w_2)}{bc(\overline{w_1}-\overline{w_2})+(w_1-w_2)}\\
&=\frac{2 bc (ad-bc)+(b+c)(bc(a+d)-ad(b+c))}{bc((a+d)-(b+c))+(bc(a+d)-ad(b+c))}\,.
\end{align*}
Hence
\begin{align*}
\frac{|w_2-t_2|^2}{1-|t_2|^2}
&=\frac{(w_2-t_2)(\overline{w_2}-\overline{t_2})}{1-t_2 \overline{t_2}}\\
&=-\frac{(a-b)(c-d)((a+d)-(b+c))(bc(a+d)-ad(b+c))}{(a-c)(b-d)(ac-bd)^2}\,.
\end{align*}

In a similar way, we have
$$t_4=L[w_1,w_2]\cap L[a,d]=\frac{2 ad (ad-bc)+(a+d)(bc(a+d)-ad(b+c))}{ad((a+d)-(b+c))+(bc(a+d)-ad(b+c))}$$
and hence
$$\frac{|w_2-t_4|^2}{1-|t_4|^2}=\frac{|w_2-t_2|^2}{1-|t_2|^2}\,.$$

By \eqref{myrho}, we have
$$\sh^2\frac12{\rho_{\mathbb{B}^2}(w_2,t_2)}=\sh^2\frac12{\rho_{\mathbb{B}^2}(w_2,t_4)}=\frac{((a+d)-(b+c))(bc(a+d)-ad(b+c))}{(a-c)(a-d)(b-c)(b-d)}\,.$$

Since
$$\sh^2\frac12{\rho_{\mathbb{B}^2}(t_2,t_4)}=\frac{4((a+d)-(b+c))(bc(a+d)-ad(b+c))}{(a-d)^2(b-c)^2}\,,$$
%$$\rho_{\B}(w_2,t_3)=\rho_{\B}(w_2,t_1),$$
we obtain
$$\frac{{\rm sh}\,\frac12\rho_{\B}(w_2,t_2)+{\rm sh}\,\frac12\rho_{\B}(w_2,t_4)}{{\rm sh}\,\frac12\rho_{\B}(t_2,t_4)}=\sqrt{\frac1{|a,b,c,d|}}={\rm th}\,\frac{d_1}2\,.$$

A similar argument yields
$$\frac{{\rm sh}\,\frac12\rho_{\B}(w_2,t_1)+{\rm sh}\,\frac12\rho_{\B}(w_2,t_3)}{{\rm sh}\,\frac12\rho_{\B}(t_1,t_3)}=\sqrt{\frac1{|d,a,b,c|}}={\rm th}\,\frac{d_2}2\,.$$

This completes the proof.
\hfill $\square$
\end{nonsec}
%===========================================================================================

%===========================================================================================
\begin{rem}
From the proof of Theorem \ref{euclamb}, it is easy to see that
$$\rho_{\B}(w_2,t_1)=\rho_{\B}(w_2,t_3)\,\,\,{\rm and}\,\,\,\rho_{\B}(w_2,t_2)=\rho_{\B}(w_2,t_4)\,.$$
\end{rem}
%===========================================================================================

%===========================================================================================
\begin{nonsec}{\bf Proof of Theorem  \ref{myhmidp}.}
Let $w=J^*[a,c]\cap J^*[b,d]$ and $w_2=L[a,c]\cap L[b,d]$, see Figure~\ref{cyclicfig31} or Figure~\ref{cyclicfig311b}.
By \eqref{mysympt}, we have
\begin{equation*}%\label{wwbar}
 \begin{cases}
{\displaystyle {w+\frac{1}{\overline{w}}=\frac{2(ac-bd)}{a-b+c-d}} } &\\
{\displaystyle { \frac{w}{\overline{w}}=\frac{ac(b+d)-bd(a+c)}{a-b+c-d}}\,,}&
 \end{cases}
\end{equation*}
and hence
\begin{equation*}\label{wwbar}
 w_2=\frac{2w}{1+|w|^2}\,.
\end{equation*}
A simple calculation yields
\begin{equation*}\label{wwbar}
 w=\frac{w_2}{1+\sqrt{1-|w_2|^2}}\,.
\end{equation*}

%==========================================================================
\begin{figure}[H]
\centering
\includegraphics[width=8cm]{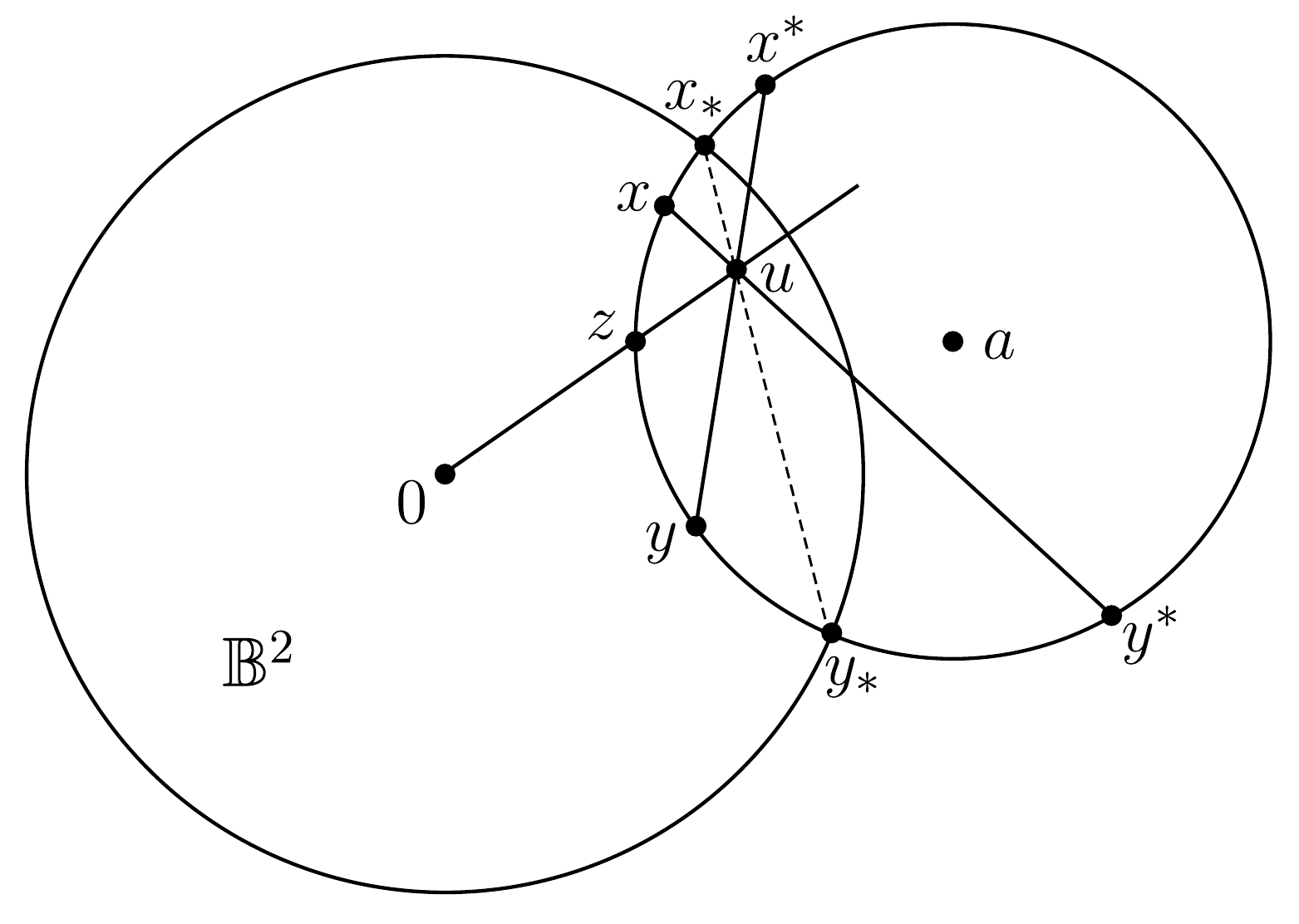}
\caption{\label{bgsfig43b} The hyperbolic midpoint $z$ of $J[x,y]$ is also the hyperbolic midpoint of $J[0,u]$. }
\end{figure}
%==========================================================================

By Remark \ref{myhmidps}, the point $w$ is the hyperbolic midpoint of $J[0,w_2]$.
By \cite[Proposition 3.1]{vw2}, the hyperbolic midpoint $z$ of $J[x,y]$ is also the hyperbolic midpoint of $J[0,u]$ (see Figure~\ref{bgsfig43b}), where
\begin{equation}\label{point-u}
u= LIS[x, y^* ,y, x^*]=LIS[0, z, x_*, y_*]= \frac{y(1-|x|^2)+x(1-|y|^2)}{1-|x|^2|y|^2}\,,
\end{equation}
see \cite[Lemma 4.6(2)]{vw1}.
Then
\begin{equation*}\label{wwbar}
z=\frac{u}{1+\sqrt{1-|u|^2}}\,.
\end{equation*}
These formulas together with \eqref{myahl} imply \eqref{myzformua}.
\hfill $\square$
\end{nonsec}
%===========================================================================================

%===========================================================================================
\begin{rem}
Let $0,x,y$ be non-collinear points in the unit disk $\mathbb{B}^2\,$ and let $z$ be the hyperbolic midpoint of the hyperbolic geodesic segment joining $x$ and $y\,.$
As shown in \cite[Fig.12-16]{vw1}, the five points
$$k=LIS[x_*,x^*,y_*,y^*]\,, \quad v=LIS[x,x_*,y,y_*]\,, \quad s=LIS[x,y_*,y,x_*]\,, $$
$$u=LIS[x,y^*,y,x^*]\,, \quad t=LIS[x_*,y^*,y_*,x^*]$$
are all on the line through the origin  and the point $z\,.$
%==========================================================================
\begin{figure}[H]
\centering
\includegraphics[width=8cm]{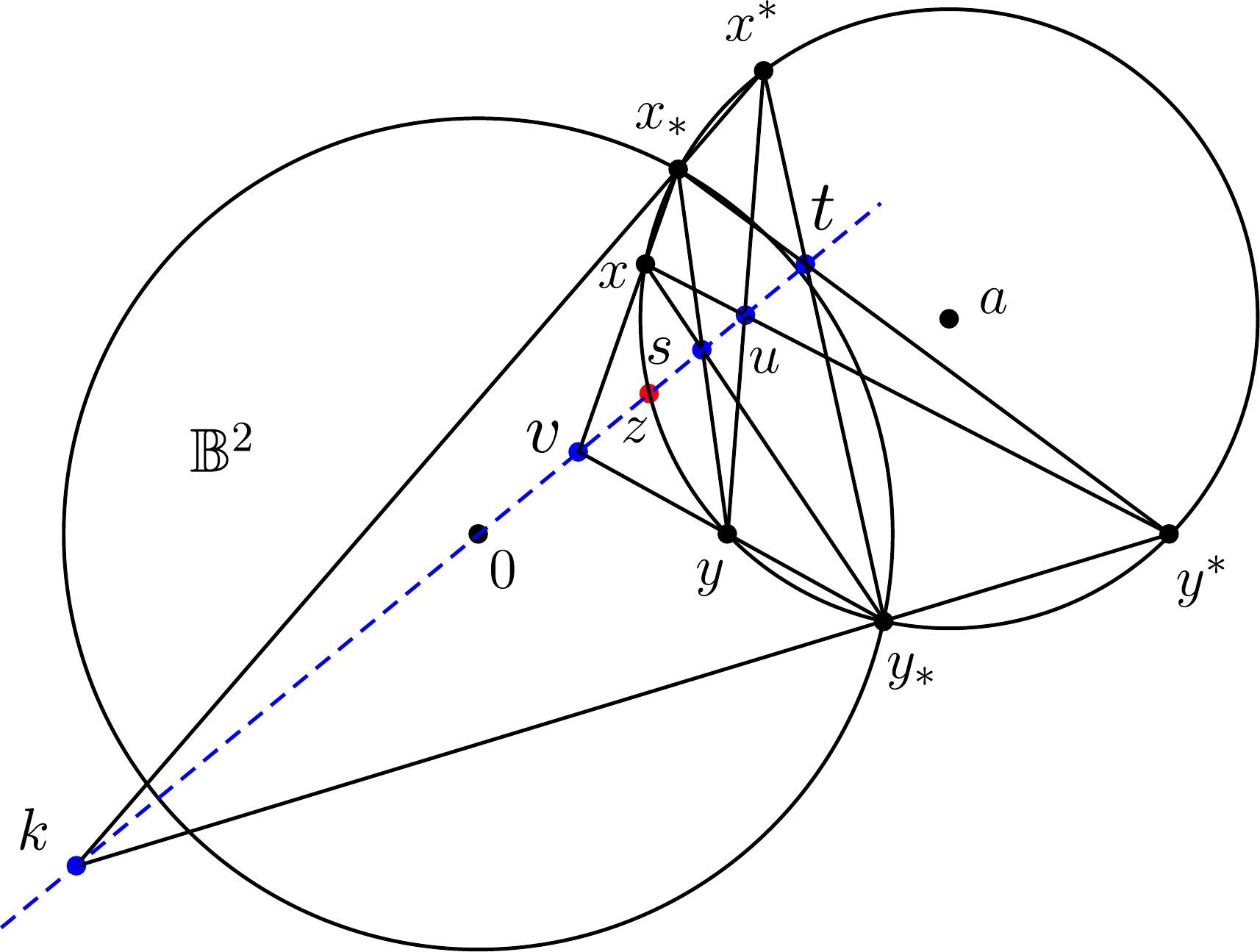}
\caption{\label{cyclicsymwholeb} Five collinear points.}
\end{figure}
%==========================================================================
%Therefore, each of these five points is of the form $tz, t \in \mathbb{R}\setminus \{0\}\,.$
As the proof of Theorem \ref{myhmidp} shows, we have an explicit formula for $z$ in terms of $u$,
but finding similar formulas for the other points seems to lead to tedious computations.
\end{rem}
%===========================================================================================

%===========================================================================================
\begin{rem}
In this paper we treated problems in hyperbolic geometry analytically. However, as the referee
pointed out  to us, there is an approach to hyperbolic
geometry that is fully analogous to the common vector space approach to Euclidean geometry.
Indeed, Ungar \cite{ungar1,ungar2} introduced the notion of a gyrovector space
and developed its theory and applications in many papers and books, see \cite{ungar1,ungar2, ungar3} for further
references. The gyrovector formalism provides a framework
 which enables one  to operate with hyperbolic geometric objects in the style of linear algebra.
In particular, there is  an elegant formula for the gyromidpoint $z$ of two points $x,y\in \mathbb{B}^2$ (see \cite[(6.91)]{ungar1} or \cite{ungar3}):
$$z=\frac{1}{2}\otimes(x\boxplus y),$$
where $\frac{1}{2}\otimes x$ is defined as $\frac{1}{2}\otimes x \oplus \frac{1}{2}\otimes x =x$, and
 $x\boxplus y=x\oplus[(\ominus x\oplus y)\oplus x]$ with the M\"obius addition $x\oplus y=\frac{x+y}{1+\bar{x}y}$ and $\ominus x=-x$.
 Calculations yield that
 \begin{equation}\label{gyro-half}
 \frac{1}{2}\otimes x=\frac{x}{1+\sqrt{1-|x|^2}}
 \end{equation}
 and
 \begin{equation}\label{gyroadd}
 x\boxplus y= \frac{y(1-|x|^2)+x(1-|y|^2)}{1-|x|^2|y|^2}.
 \end{equation}
 Combining \eqref{gyro-half} and \eqref{gyroadd}, we obtain
 $$\frac{1}{2}\otimes(x\boxplus y)=\frac{y(1-|x|^2) + x(1-|y|^2)}{1-|x|^2|y|^2 + A[x,y] \sqrt{(1-|x|^2)(1-|y|^2)}},$$
 which is coincident with the formula in Theorem \ref{myhmidp}.
 Hence we conclude that the gyromidpoint of $x$ and $y$ is the hyperbolic midpoint of these two points.
 According to the formulas \eqref{point-u} and \eqref{gyroadd}, we see that $u=x\boxplus y$ and our Figure \ref{bgsfig43b} gives a Euclidean geometric construction for the gyroaddition $x\boxplus y$ (also see \cite{vw1}).
\end{rem}
%===========================================================================================

\begin{comment}
%===========================================================================================
\begin{rem}
As shown in \cite[Figure12-16]{vw1}, the points
$LIS[x,y^*,y,x^*]\,,$ $LIS[x,x_*,y,y_*]\,,$ $LIS[x,y_*,y,x_*]\,,$ $LIS[x_*,y^*,y_*,x^*]\,,$ $LIS[x_*,x^*,y_*,y^*]$
are all on the line through the origin and the hyperbolic midpoint of the hyperbolic geodesic segment joining $x$ and $y$,
but finding explicit formulas for these five points comparable to Theorem \ref{myhmidp} seems to lead tedious computation.
%In \cite{vw1}, the authors provided five constructions based on five collinear points to find the midpoint of the hyperbolic geodesic segment joining two given points in $\B$.
%In the proof of Theorem  \ref{myhmidp}, we used the point $u$ which is one of the five collinear points and has the special property $\rho_{\B}(0,u)=2\rho_{\B}(0,z)$ while the %other four collinear points do not have such simple property.
\end{rem}
%===========================================================================================
\end{comment}

%%%%%%%%%%%%%%%%%%%%%%%%%%%%%%%%%%%%%
%%%%%%%%%%%%%%%%%%%%%%%%%%%%%%%%%%%%%
%%%%%%%%%%%%%%%%%%%%%%%%%%%%%%%%%%%%%
\section{Normalisation of quadruples}
%%%%%%%%%%%%%%%%%%%%%%%%%%%%%%%%%%%%%
%%%%%%%%%%%%%%%%%%%%%%%%%%%%%%%%%%%%%
%%%%%%%%%%%%%%%%%%%%%%%%%%%%%%%%%%%%%

Quadruples of points $a,b,c,d\in \overline{\mathbb{C}}$ in general positions are often more convenient to handle
if they are normalised, brought by a preliminary M\"obius transformation to a canonical position.
The most commonly used reduction to a canonical position is to
find a M\"obius transformation $h$ such that the points are mapped like this
\begin{equation}\label{cano1}
(a,b,c,d)\mapsto (h(a),h(b),h(c), h(d))= (0,1,p,\infty)\,,   \quad p>1\,.
\end{equation}
%and $p$ is uniquely determined if we require $Im(p) >0\,.$
Another possibility is to
find a M\"obius transformation $g$ mapping the points to positions symmetric with respect to the origin (see e.g., \cite[Lemma 7.24]{avv} for a formula for $y$)
\begin{equation}\label{cano2}
(a,b,c,d)\mapsto (g(a),g(b),g(c), g(d))= (-1,-y,y,1)\,,   \quad |y|\le 1\,.
\end{equation}

Both of  these normalisations \eqref{cano1} and \eqref{cano2} preserve the cross ratio of the quadruples,
but in general change euclidean or chordal distances.
We can measure the change of distances utilising the Lipschitz constant \eqref{myLip1}.
For each of the canonical forms we can indicate the "cost" of normalisation as the Lipschitz constant of the normalising map.
The Lipschitz constant of every M\"obius transformation is a finite number \cite[Theorem 3.6.1]{b}.

Here we will study the normalisation in the special case
when the points of the quadruple are on the unit circle
and we will emphasise a symmetrization procedure from euclidean geometry.
We have two canonical positions:
the normalised quadruple is
(a) symmetric with respect to a diameter of the unit disk
or
(b) symmetric with respect to the origin.
It is clear that if (b) holds, then also (a) holds.

To see the existence of such kind of normalisation, we can use an appropriate hyperbolic rotation around the point of intersection of the hyperbolic lines which is a composition of two canonical mappings of the unit disk and a rotation around the origin as shown in Figure~\ref{cyclicsymwholeb}.

%==========================================================================
\begin{figure}[H]
\centering
\includegraphics[width=16cm]{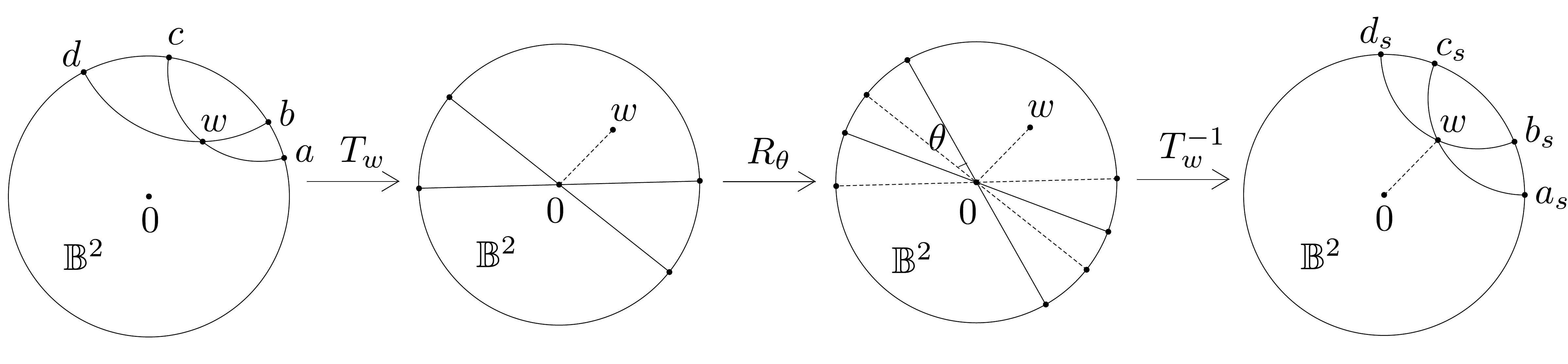}
\caption{\label{cyclicsymwholeb} Hyperbolic rotation}
\end{figure}
%==========================================================================

As an application of the results from Section \ref{go}, mainly from Theorem \ref{myoccen} and Theorem \ref{mythoccen},
we provide a euclidean construction of the symmetrization of random four points $a\,,b\,,c\,,d\in\partial\B$ with respect to a diameter, see Figure \ref{cyclicsymfiga} and Figure \ref{cyclicsymfigb}.

%==========================================================================
\begin{figure}[h]
\begin{minipage}[t]{0.45\linewidth}
\centering
\includegraphics[width=7.2cm]{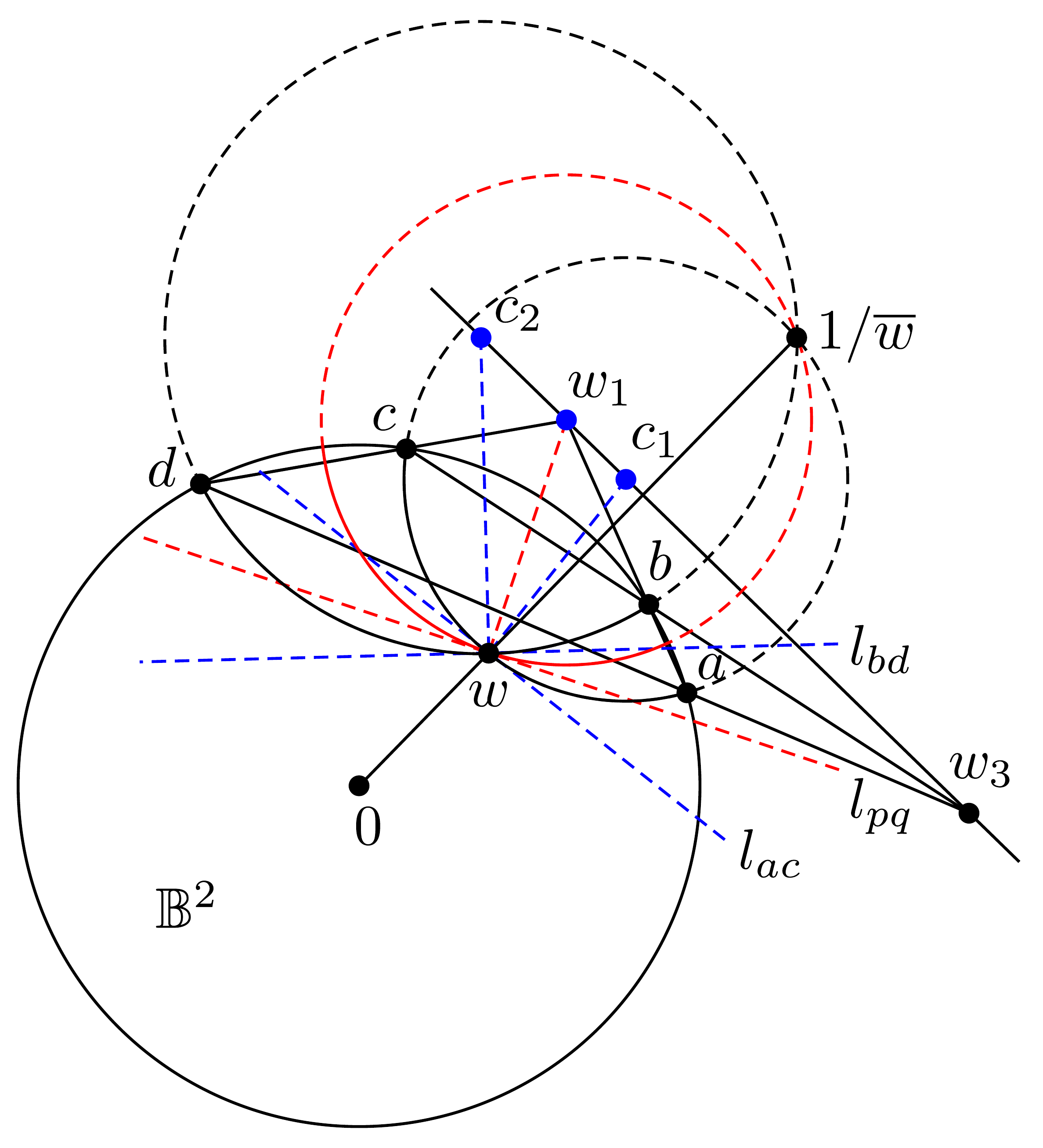}
\caption{\label{cyclicsymfiga}}
\end{minipage}
\hfill
\begin{minipage}[t]{0.45\linewidth}
\centering
\includegraphics[width=7cm]{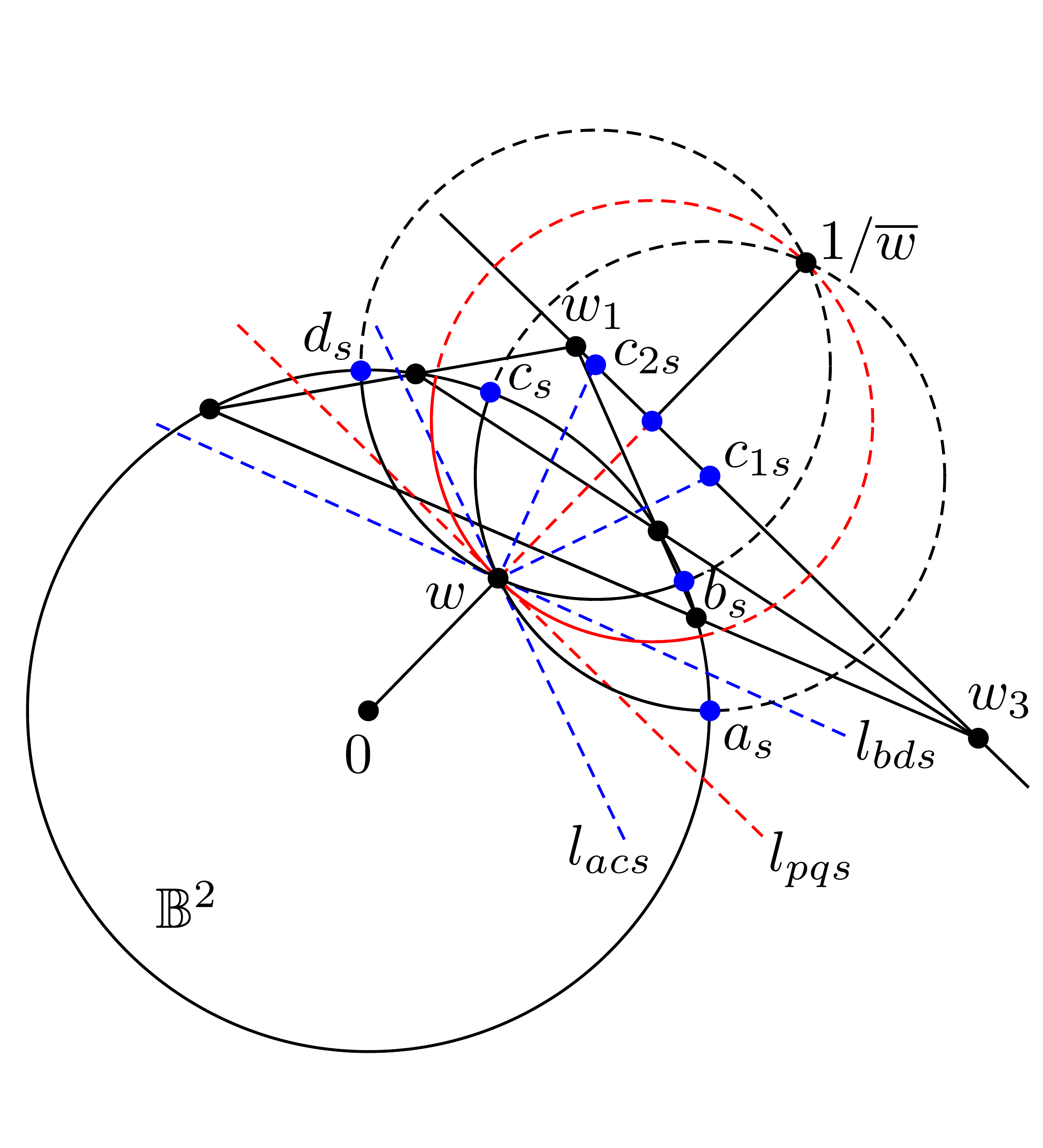}
\caption{\label{cyclicsymfigb}}
\end{minipage}
\end{figure}
%==========================================================================

%===========================================================================================
\begin{enumerate}
\item Find the points $w_1=LIS[a,b,c,d]$, $w_3=LIS[a,d,b,c]$, and $w=J^*[a,c]\cap J^*[b,d]$;
\item Construct the lines $L[w_1,w_3]$ and $L[0,w]$;
%\item Construct the circle $Y$ centered at $w_1$ with radius $|w_1-w|$ intersecting the unit circle at points $p\,,q$;
\item Find the centers $c_1$ and $c_2$ of the circles $C[a,c]$ and $C[b,d]$ on $L[w_1,w_3]$, respectively;
\item Construct the lines $l_{ac}\,,l_{pq}\,,l_{bd}$ passing through the point $w$ such that they are perpendicular to $[c_1,w]\,,[w_1,w]\,,[c_2,w]$, respectively;
\item Rotate the three lines $l_{ac}\,,l_{pq}\,,l_{bd}$ with respect to the point $w$ to the lines $l_{acs}$, $l_{pqs}$, $l_{bds}$, respectively,
 such that $l_{pqs}$ is perpendicular to $L[0,w]$;
\item Find the points $c_{1s}$ and $c_{2s}$ on $L[w_1,w_3]$ by constructing the lines passing through $w$ and orthogonal to $l_{acs}$ and $l_{bds}$, respectively;
\item Construct the circle centered at $c_{1s}$ with radius $|c_{1s}-w|$ intersecting the unit circle at points $a_s$ and $c_s$;
\item Construct the circle centered at $c_{2s}$ with radius $|c_{2s}-w|$ intersecting the unit circle at points $b_s$ and $d_s$.
\end{enumerate}

Then the points $a_s\,,b_s\,,c_s\,,d_s$ are the symmetric points of $a\,,b\,,c\,,d$. Specifically, the points $a_s$ and $d_s$ are symmetric with respect to the diameter through $w$, and so are $b_s$ and $c_s$.
%===========================================================================================

\medskip

%%%%%%%%%%%%%%%%%%%%%%%%%%%%%%%%%%%%%%%%%%%%%%%%%%%%%%%
\subsection*{Acknowledgments.}
This work was supported by the National Natural Science Foundation
of China (NNSFC) under Grant No.11771400 and No.11601485,
the joint mobility project of the National Natural Science Foundation of China and the Academy of Finland under Grant No.11911530457,
and Science Foundation of Zhejiang Sci-Tech University (ZSTU) under Grant No.16062023-Y. G. Wang and X. Zhang are indebted to the
Department of Mathematics and Statistics at University of Turku for its hospitality during their stay there.
M. Vuorinen is indebted to the
Department of Mathematical Sciences at Zhejiang Sci-Tech University for its hospitality during his visit there.
The authors also thank the anonymous referee for his/her valuable comments and suggestions.
%%%%%%%%%%%%%%%%%%%%%%%%%%%%%%%%%%%%%%%%%%%%%%%%%%%%%%%

\medskip

%%%%%%%%%%%%%%%%%%%%%%%%%%%%%%%%%%%%%%%%%%%%%%%%%%%%%%%
\bibliographystyle{siamplain}
%\bibliography{Bref}

\begin{thebibliography}{CHKV}

%=========================================================================================
\bibitem[AVV]{avv}
\textsc{G.\,D. Anderson, M.\,K. Vamanamurthy, and M. Vuorinen},
Conformal Invariants, Inequalities, and Quasiconformal Maps,
John Wiley \& Sons, New York, 1997.

%=========================================================================================
\bibitem[B]{b}
{\sc  A.\,F. Beardon},
The Geometry of Discrete Groups,
Graduate Texts in Math., Vol. 91, Springer-Verlag, New York, 1983.

%\bibitem[M]{m} {\small
%{\sc  Mathematica }: Circles20180401. Notebook}

%=========================================================================================
\bibitem[G]{g}
\textsc{C. Goodman-Strauss},
{\it Compass and straightedge in the Poincar\'e disk},
Amer. Math. Monthly 108 (2001), 38--49.


%=========================================================================================
\bibitem[KV]{kv}
\textsc{R. Kl\'en and M. Vuorinen},
{\it Apollonian circles and hyperbolic geometry}, 
J. Analysis 19 (2011), 41--60.

%=========================================================================================
\bibitem[U1]{ungar1}
\textsc{A. A. Ungar}, 
Analytic Hyperbolic Geometry: Mathematical Foundations and Applications,
World Scientific Publishing Co. Pte. Ltd., Hackensack, NJ, 2005.

%=========================================================================================
\bibitem[U2]{ungar2}
\textsc{A. A. Ungar}, 
{\it From M\"obius to gyrogroups},
 Amer. Math. Monthly 115 (2008), 138--144.

%=========================================================================================
\bibitem[U3]{ungar3}
\textsc{A. A. Ungar}, 
{\it  Midpoints in gyrogroups}, 
Found. Phys. 26 (1996), 1277--1328.

%=========================================================================================
\bibitem[VW1]{vw2}
{\sc  M. Vuorinen and G. Wang},
{\it Hyperbolic Lambert quadrilaterals and quasiconformal maps,}
 Ann. Acad. Sci. Fenn. Math. 38 (2013), 433--453.
%http://arxiv.org/abs/1203.6494
%doi:10.5186/aasfm.2013.3845,
%{arXiv:1203.6494 math.CA}}

%=========================================================================================
\bibitem[VW2]{vw1}
{\sc M. Vuorinen and G. Wang},
{\it Bisection of geodesic segments in hyperbolic geometry,}
Complex Analysis and Dynamical Systems V, 273--290,
Contemp. Math., 591, Israel Math. Conf. Proc., Amer. Math. Soc., Providence, RI, 2013.

%=========================================================================================
\bibitem[W]{w}
{\sc  G. Wang},
{\it The adjacent sides of hyperbolic Lambert quadrilaterals,}
Bull. Malays. Math. Sci. Soc. 41 (2018), 1995--2010.

\end{thebibliography}

%%%%%%%%%%%%%%%%%%%%%%%%%%%%%%%%%%%%%%%%%%%%%%%%%%%%%%%

\end{document}